\documentclass{amsart}
\usepackage{graphicx}
\usepackage{amssymb}
\usepackage{amsfonts}
\usepackage{hyperref}
\usepackage{mathrsfs}
\swapnumbers
\newtheorem{theorem}{Theorem}[section]

\newtheorem{proposition}[theorem]{Proposition}
\theoremstyle{definition}

\newtheorem{assumption}[theorem]{Assumption}
\newtheorem{remark}[theorem]{Remark}
\newtheorem{example}[theorem]{Example}

\numberwithin{equation}{section}
\theoremstyle{plain}

\numberwithin{equation}{section} %% Comment out for sequentially-numbered
\numberwithin{figure}{section} %% Comment out for sequentially-numbered
\theoremstyle{plain}
\theoremstyle{plain}
\theoremstyle{remark}
\newtheorem*{acknowledgement*}{Acknowledgement}
\theoremstyle{example}

\newcommand{\cB}{{\mathcal B}}
\newcommand{\cC}{{\mathcal C}}

\newcommand{\cE}{{\mathcal E}}
\newcommand{\cF}{{\mathcal F}}

\newcommand{\cH}{{\mathcal H}}

\newcommand{\cL}{{\mathcal L}}

\newcommand{\cS}{{\mathcal S}}

\newcommand{\cU}{{\mathcal U}}

\newcommand{\cX}{{\mathcal X}}

\newcommand{\te}{{\theta}}

\newcommand{\Om}{{\Omega}}
\newcommand{\om}{{\omega}}
\newcommand{\ve}{{\varepsilon}}
\newcommand{\del}{{\delta}}
\newcommand{\Del}{{\Delta}}
\newcommand{\gam}{{\gamma}}
\newcommand{\Gam}{{\Gamma}}

\newcommand{\sig}{{\sigma}}
\newcommand{\al}{{\alpha}}
\newcommand{\be}{{\beta}}
\newcommand{\ka}{{\kappa}}
\newcommand{\la}{{\lambda}}

\newcommand{\bbC}{{\mathbb C}}
\newcommand{\bbE}{{\mathbb E}}

\newcommand{\bbN}{{\mathbb N}}

\newcommand{\bbR}{{\mathbb R}}
\newcommand{\bbT}{{\mathbb T}}
\newcommand{\bbZ}{{\mathbb Z}}

\begin{document}
\title[]{A vector-valued almost sure invariance principle for time dependent non-uniformly expanding dynamical systems} %A sequential RPF theorem and its applications to inhomogeneous ...
 \vskip 0.1cm
 \author{Yeor Hafouta \\
\vskip 0.1cm
Department of Mathematics\\
The Ohio State University\\
}%
\address{
Department of Mathematics, The Ohio State University University}
\email{yeor.hafouta@mail.huji.ac.il, hafuta.1@osu.edu}%

\thanks{ }
\dedicatory{  }
 \date{\today}

%\footnotetext[1]{}
\maketitle
\markboth{Y. Hafouta}{A vector-valued ASIP} %with non-expanding base maps}
\renewcommand{\theequation}{\arabic{section}.\arabic{equation}}
\pagenumbering{arabic}

\begin{abstract}\noindent
We prove a vector-valued almost sure invariance principle for some classes of time dependent non-uniformly distance expanding dynamical systems. The models we have in mind are certain sequential versions of the smooth non-uniformly distance expanding maps considered in \cite{castro} and \cite{Vara}, as well as certain types of sequences of covering maps. 
Our results rely on the theory of complex projective metrics which was developed in \cite{Rug}, together with the spectral methods of Gou\"ezel \cite{GO}. A big advantage in applying the theory of complex cones here is that it also yields additional  probabilistic limit theorems for random dynamical systems, as described at the last section of this paper. 
\end{abstract}

\section{Introduction}\label{sec1}
%(see \cite{Neg1}, \cite{Neg2}, \cite{GH} and \cite{HH}). 
%Mention Raughi, we obtain a local limit theorem

Probabilistic limit theorems for deterministic dynamical systems  is a well studied topic. One important generalization of such results (see, for insatance \cite{Kifer-1996} and \cite{Kifer-1998}) is to random dynamical systems in which the system evolves according to iterates of random transformations of the form $T_{\te^{n-1}\om}\circ\cdots\circ T_{\te\om}\circ T_\om,\,\om\in\Om$, where $(\Om,\cF,P,\te)$ is some ergodic  invertible measure preserving system, which can be viewed as a  ``driving process". The central limit theorem (CLT) for partial sums generated by random dynamical systems has been studied by many authors. Recently, finer results such as the local central limit theorem (LCLT) and the Berry-Esse\'en theorem (optimal convergence rate in the CLT) have been obtained for several classes of random uniformly distance expanding and hyperbolic  dynamical systems (see \cite{drag}, \cite{drag1} and Ch.7 of \cite{book}). These results rely on certain types of analysis of complex transfer operators, and they did not cover, for instance random non-uniformly distance expanding maps. 

A related, but more general, setup is the case when the underlying sequence of random variables has the form
$X_n=T_0^n\textbf{x}_0$,  where $\textbf{x}_0$ is some random variable and $T_0^n=T_{n-1}\circ T_{n-2}\circ\cdots\circ T_0$  for some given sequence of  maps $T_0,T_1,T_2,...$. Results in this direction where obtained, for instance, in \cite{Arno}, \cite{Conze}, \cite{Zemer} and \cite{Nicol} and references therein. The setup of random dynamics is a special case of this setup, where $T_j=T_{\te^j\om}$ are stationary random maps. In  \cite{dragASIP} and \cite{Hyde} the authors proved an almost sure real-valued invariance principle (ASIP) for random and sequential dynamical systems, which means that the underlying partial Birkhoff sums can be approximated by a sum of independent Gaussian random variables with an error term which is smaller than the square root of the variance of the partial sum (such an estimates yields the law of iterated logarithm). Both papers invoked a recent result on ASIP  for ``reverse" martingales due to C. Cuny and F. Merlev\`{e}de \cite{CM}, and assumed that the underlying transfer operators preserve the same probability measure, which essentially means the exponent of the underlying potential function is the inverse of the Jackobian. We stress that these results were obtained for real-valued observables, for which the results in \cite{CM} apply.

In this paper we will prove an ASIP with for parital sums of the form $S_n(x)=\sum_{j=0}^{n-1}u_j\circ T_0^j(x)$, where $u_j$ is seqeunce of vector-valued H\"older continuous or differentiable functions and $x$ is distributed accodring to a special measure $\mu_0$, and, for instance, each $T_j$ is a non-uniformly distance expanding map satisfying the conditions in \cite{castro} and \cite{Vara}. Our results hold true 
when the covariance matrix of $S_{j,n}=\sum_{k=0}^{n-1}u_{j+k}\circ T_{j}^{k}$ grows linearly fast  in $n$ uniformly in $j$, a condition which we verify in the case when all the maps $T_j$ and  the functions $u_j$ lie in some neighborhood of appropriate map $T$ and a function $u$ (see also Remark \ref{Rem1}).
 Even when all the maps $T_j$ coincide with the same map $T$ from \cite{castro} our results are new, and in this case they yield an ASAP for the sums $S_n(x)$ when $x$ is distributed according to one of the equilibrium states constructed in \cite{castro} (e.g. the unique measure with maximal entropy). We want to stress that even in the deterministic setup of \cite{castro}, it is unclear in which circumstances one can choose the underlying potential to be the inverse of the Jacobian of $T$.

In \cite{YeorDavor} an ASIP for random hyperbolic and uniformly distance expanding maps has been derived using a certain modification of the spectral method of Gou\"ezel \cite{GO} for non-stationary sequences, and the vector-valued obtained in this manuscript relies on this modification, as well. This method requires that  appropriate complex perturbations of the
the underlying sequence of complex transfer operators to have certain ``spectral" properties, which in the setup of this paper are obtained using the theory of complex projective metrics developed in \cite{Rug} (which was applied in Ch. 5 of \cite{book} with uniformly distance expanding maps).

Once the appropriate projective contraction properties of the underlying (real) transfer operator are established, in the case of real-valued observables $u_j$ and when the $T_j$'s are uniformly H\"older continuous it seems plausible that the results for reverse martingales in \cite{CM} can yield the ASIP (by using the scheme in \cite{dragASIP}), but in this paper we consider vector-valued observables, and we do not assume H\"older continuity of the underlying maps.
Moreover, using  complex projective metrics (associated with complex cones),  a Berry-Esse\'en theorem,  
and moderate and (local) large deviations principles  for random non-uniformly distance expanding dynamical systems follow, which is another advantage in using cones. For certain classes of weakly-expanding random maps (e.g. Manneville-Pomeau maps) we will also obtain a local central limit theorem, see Remark \ref{Rem2}. In fact, for such maps we are able to obtain the ASIP under weaker assumptions (see Remark \ref{Rem1}).

This paper is organized as follows. In  Section \ref{Sec2} we  describe our setup and state our main results.  Section \ref{Sec4} contains the additional tool required for the ASIP: we  obtain there a sequential Ruelle-Perron-Frobenius (RPF) theorem for an appropriate sequence of parametrized complex transfer operators (by applying contraction properties of complex cones). We will also apply these results in order to control the covariance of the underlying partial sums, which will yield a uniform control over the norm of these complex operators, an ingredient which is crucial for  applying  the aforementioned modification of Gou\"ezel's ASIP.
In Section \ref{Sec5} the proof of the main result is finalized. Finally, in Section \ref{RDS} we will obtain additional limit theorems for of real-valued observables in the case of random non-uniformly distance expanding dynamical systems.

\section{Preliminaries and main results}\label{Sec2}
We will consider in this paper two types of models of sequences of non-uniformly distance expanding maps.

\subsection{Locally smooth non-uniformly distance expanding maps}\label{Model 1}
In this section we will start from the setup in \cite{castro} and \cite{Vara}. In some sense, the model considered here is less general than the one considered in Section \ref{Model 2}, but it is more explicit and therefore we present it first.

Let $M$ be a finite dimensional compact and connected Riemnnian manifold  with distance $\rho$. Let $T:M\to M$ be a
local homeomorphism and assume that there exists a continuous function $x\to L(x)$ such that, for every $x\in M$ there
is a neighborhood $U_x$ of $x$ that $T_x:U_x\to T(U_x)$ is invertible and
\[
\rho(T_x^{-1}(y),T_x^{-1}(z))\leq L(x)d(y,z),\,\,\forall y,z\in T(U_x).
\]
In particular every point has the same finite number of preimages $\deg(T)$ which coincides with the degree of $T$. Our additional assumption is that there exist constants
$\sig>1$ and $L\geq 1$, and an open region $A\subset M$ such that
\\

(H1) $L(x)\leq L$ for every $x\in A$ and $L(x)<\sig^{-1}$ for all $x\notin A$;
\\

(H2) There exists a finite covering $U$ of $M$ by open domains of injectivity for $T$ such that A can be covered by
$q<\deg(T)$;
\\

Next, let $\phi:M\to\mathbb R$ be a $C^1$-function. 
Our further restrictions on the function $\phi$, and the constants appearing in (H1) and (H2) are summarized in the following 
\begin{assumption}\label{Bound ass}
There exists a constant $\ve>0$ so that  
\[
\sup{\phi}-\inf{\phi}\leq \ve
\]
and
\begin{equation}\label{s sup}
s:=e^{\ve}\cdot\frac{qL+(\deg(T)-q)\sig^{-1}}{\deg(T)}<1.
\end{equation}
\end{assumption}

Next, let $T_j:M\to M$ be a sequence of maps satisfying (H1) and (H2) and $\phi_j:M\to\mathbb R$ be a sequence of scalar $C^1$-functions. Let $d\geq 1$ and $u_j:M\to\mathbb R^d$ be a sequence of vector-valued $C^1$ functions. For each $j$ and $n\geq0$, consider the partial sums 
\[
S_{j,n}=\sum_{k=0}^{n-1}u_{j+k}\circ T_{j+n-1}\circ\cdots\circ T_{j+1}\circ T_j
\] 
and set $S_n=S_{0,n}$.
Our main result in the above setup is the following
\begin{theorem}\label{ASIP 1}
Suppose that (H1), (H2) and Assumption \ref{Bound ass} are satisfied.
Let $\mu_j=h_j^{(0)}d\nu_j^{(0)}$  be the probability measures from Theorem \ref{RPF SDS} and assume that there exists a constant $c>0$ so that for any sufficiently large $n$ and any $v\in\mathbb R^d$ we have
\begin{equation}\label{Uniform Cov}
\inf_{j}\text{Cov}_{\mu_j}S_{j,n} v\cdot v\geq cn|v|^2.
\end{equation}
Then there exists $\ve>0$  so that if $(T_j,\phi_j),\,j\in\bbZ$ belong to a $C^1$-ball of radius $\ve$ around $(T,\phi)$ then, for any $\del>0$ there is a coupling between $(u_j\circ T_0^j)_j$, considered as a sequence of random variables on $(M,\mu_0)$, and  a sequence of centered Gaussian random vectors $Z_1,Z_2,...$ so that 
\[
\Big|S_n-\int S_n(y)d\mu_0(y)-\sum_{j=1}^n Z_j\Big|=o(n^{\frac 14+\del}),\quad\text{almost-surely}.
\]
 Moreover, with $\cS_n=S_n-\int S_n(y)d\mu_0(y)$,
 there exists a constant $C>0$ so that for any unit vector $v\in\mathbb R^d$,
\begin{equation}\label{Var est}
\Big\|\cS_n\cdot v\Big\|_{L^2}-Cn^{\frac14+\del}\leq \Big\|\sum_{j=1}^n Z_j\cdot v\Big\|_{L^2}\leq \Big\|\cS_n\cdot v \Big\|_{L^2}+Cn^{\frac14+\del}.
\end{equation}
\end{theorem}
In Section \ref{CovSec} we will show that (\ref{Uniform Cov}) holds true if $\ve$ is small enough and the functions $u_j$ lie in some $C^1$-neighborhood of a function $u$ which is not a couboundary with respect to $T$.
We want also to stress that when considering the initial measure $\mu_0$, the correlation between the summands $u_{j+k}\circ T_{j+n-1}\circ\cdots\circ T_{j+1}\circ T_j$ converge exponentially fast to $0$.
Moreover, when $T_j=T$ and $\phi_j=\phi$ then $\mu_0$ is the unique equilibrium state corresponding to $T$ and $\phi$ constructed in \cite{castro}. 
When $\phi=0$ then we just get the unique measure with maximal entropy.  %We also remark that when $\phi_j$ is the Jackobian of $f_j$ then $\nu_0$ is just the volume measure on $M$. 
Furthermore, similarly to \cite{MSU}, it is possible to show that the measures $\nu_j$ from Theorem \ref{RPF SDS} are conformal in the sense that for any measurable set such that $T_j|A$ is injective we have
\begin{equation}\label{Conformal}
\nu_{j+1}(T_j(A))=e^{\Pi_j(0)}\int_A e^{-\phi_j(x)}d\nu_j(x)
\end{equation}
where $\Pi_j(0)$ is the logarithm of $\la_j(0)$ from Theorem \ref{RPF SDS}. This coincides with the definition of conformal measures in \cite{castro} in the case of a single map $T$ and a potential $\phi$. It is also important to note that, in the random dynamical system case considered in Section \ref{RDS}, the measure $\mu_0=\mu_\omega$ is the disintegration of a probably measure on the skew-product space, which is invariant and ergodic with respect to the skew product map (namely, it is the so called, random Gibbs measure, see Section \ref{RDS}). 
 %talk about certain decay of correlations that we have here...

\subsubsection{Examples}
In this section we will give several examples for maps $T$ and $T_j$ (most of them are discussed in \cite{castro} and \cite{Vara}).

\begin{example}\label{Ex1.0}
Consider an interval map $g:[0,1)\to [0,1)$ of the form $g(x)=mx\text{ mod}1$. Take a small open subinterval of each monotonicity interval of $g$, and perturb $g$ on this interval in such a way that the resulting new inverse branch will have derivatives smaller than $1$ at some points. Denote by $T:\mathbb T\to\mathbb T$ the resulting map. There are many ways to construct such maps $T$, and we can consider a sequence $T_j$ of such perturbations of $g$. The ASAP holds true if all of these perturbations are sufficiently close to $g$, say, and the functions $\phi_j$ have sufficiently small oscillation.
\end{example}

The following example comes from \cite{castro}.

\begin{example}\label{Ex1.1}
Let $g:\mathbb T^d\to\mathbb T^d$ be a linear expanding map. Fix some covering $\cU$ by domains of injectivity for $g$
and some $U_0\in U$ containing a fixed (or periodic) point $p$. Then deform $g$ on a small neighborhood $P_1$ of $p$ inside
$U_0$ by a pitchfork bifurcation in such a way that $p$ becomes a saddle for the perturbed local diffeomorphism $T$ . In
particular, such perturbation can be done in the $C^r$-topology, for every $r>0$. By construction, $T$ coincides with $g$ in
the complement of $P_1$, where uniform expansion holds. Observe that we may take the deformation in such a way that
$T$ is never too contracting in $P_1$, which guarantees that conditions (H1) and (H2) hold, and that $T$ is still topologically
exact. Assumption \ref{Bound ass}
 is clearly satisfied by any $C^1$-potential with a sufficiently small oscillation. 
\end{example}

\begin{example}[Manneville-Pomeau map]\label{MP map}
For each  $\beta\in(0,1)$, let $f_\beta:[0,1]\to[0,1]$ be the $C^{1+\beta}$-local diffeomorphism
given by
\begin{equation*}
f_\beta(x)=
\begin{cases}
x(1+2^\be x^\beta)\text{ if }0\leq x\leq\frac12\\
2x-1\text{ if }\frac 12<x\leq 1\\
\end{cases}
\end{equation*}
Then our assumptions hold true if  $T=f_\be$ for some $\be$ and $\phi$ is a $C^1$-function with sufficiently small oscillation. We refer to Example \ref{Ex2.1} which includes the case when each $T_j$ is a  Manneville-Pomeau map and $\phi_j=-t_j\log J(f_j)$ for a sufficiently small $t_j$ (in the setup there $\phi_j$ does not have to be a $C^1$-function).
\end{example}

We also want to mention the unimodel map $T(x)=-\frac 18 x(x-1)(x+\frac 18)$ considered in Example 2.5 in \cite{Vara}. For instance, in these circumstances, we can consider the simple case when $T_j=T$ and $\phi_j=0$ for each $j$, but $u_j$ may depend on $j$. This shows that a vector-valued ASIP holds true when $\mu_0=\mu$ is the unique measure of maximal entropy corresponding to $t$. When $(T_j,\phi_j)$ are only close to $(T,0)$ then $\mu_0$ is a certain perturbation of $\mu$, and we still get the ASIP with this initial measure.

\subsection{Non-uniformly ``expanding" covering maps}\label{Model 2}
Let $(\cX_j,\rho_j),\,j\in\bbZ$ be a two sided sequence of bounded metric spaces, noramlized in size so that $\text{diam}(\cX_j)\leq 1$, and let $T_j:\cX_{j}\to\cX_{j+1}$ be a sequence of maps satisfying the following

\begin{assumption}\label{Pairing ass}
There exist two sided sequences $(L_j)$, $(\sig_j)$, $(q_j)$ and $(d_j)$ so that $L_j\leq L$ for some $L\geq1$ and for each $j$ we have
 $\sigma_j>1$, $L_j\geq 1$, $q_j,d_j\in\bbN$, $q_j<d_j$ and for any $x,x'\in\cX_{j+1}$ we can write 
\[
T_j^{-1}\{x\}=\{x_1,...,x_{d_j}\}\,\,\text{ and }\,\,T_j^{-1}\{x'\}=\{x'_1,...,x'_{d_j}\}
\]
where for any $i=1,2,...,q_j$,
\[
\rho_j(x_i,x'_i)\leq L_j\rho_{j+1}(x,x')
\]
while for any $i=q_{j}+1,...,d_j$,
\[
\rho_{j}(x_i,x'_i)\leq \sig_j^{-1}\rho_{j+1}(x,x').
\]
\end{assumption}
An immediate example is the case when all the maps $T_j$ satisfy Assumptions (H1) and (H2) uniformly in $j$ (see the proof of Theorem 5.1 in \cite{castro} applied with each $T_j$ separately), but there are different examples (see Section \ref{Examples 2}).

Next, let $\alpha\in(0,1]$ and let $\phi_j$  be a sequence of bounded real-valued H\"older functions on $\cX_j$ with exponent $\alpha$. Denote by $\cH_j$ the space of such functions equipped with the norm 
\[
\|g\|=\|g\|_\infty+v(g)
\]
where $\|g\|_\infty=\sup|g|$ and $v(g)$ is the smallest number so that $|g(x)-g(y)|\leq v(g)\big(\rho_j(x,y)\big)^\al$ for any $x$ and $y$ in $\cX_j$.
In the case when $\al=1$ and each $\cX_j$ is a Riemannian manifold we will also consider the norms $\|g\|=\|g\|_{C^1}=\sup|g|+\sup\|Dg\|$ on the space of $C^1$-functions, namely $v(g)$ above is replaced by the supremum norm of the deferential of $g$ (so in this case $v(g)$ could either be the Lipschitz constant or $\sup\|Dg\|$). 
Our additional requirements  from the function $\phi_j$ are summarized in the following

\begin{assumption}\label{Bound ass2}
We have $\sup_j\|\phi_j\|<\infty$,
\[
\sup{\phi_j}-\inf{\phi_j}\leq \ve_j\,\,\text{ and }\,\sup_j\sup_x\sum_{y\in T_j^{-1}\{x\}}e^{\phi_j(y)}<\infty,
\]
where $\ve_j$ is a sequence of positive constants satisfying
\begin{equation}\label{s sup2}
s:=\sup_{j}e^{\ve_j}\cdot\frac{q_jL_j^\al+(d_j-q_j)\sig_j^{-\al}}{d_j}<1.
\end{equation}
\end{assumption}
The inequality (\ref{s sup}) is a quantitative estimate on the amount of contraction is allowed, given the amount of expansion $T_j$ has, and the oscillation of $\phi_j$.

Next, let $d\in\bbN$ and $u_j:\cX_j\to\mathbb R^d$ be a sequence of vector-valued functions so that $u_j\in\cH_j$ for each $j$ and the sequence of norms $\|u_j\|$ is bounded in $j$. For each $n$ and $j$ set
\[
S_{j,n}u=\sum_{i=0}^{n-1}u_{i+j}\circ T_{j}^i, 
\]
where $T_j^n=T_{j+n-1}\circ T_{n-2}\circ\cdots\circ T_j$. 
Let $\mu_j=h_j^{(0)}d\nu_j^{(0)}$ be the sequence of equivariant Gibbs measures (i.e. $(T_j)_*\mu_j=\mu_{j+1}$) constructed in Theorem \ref{RPF SDS} (so that (\ref{Exponential convergence}) and (\ref{ExpDec}) hold true).
We note the measures $\nu_j^{(0)}$ satisfy all the properties discussed after Theorem \ref{ASIP 1} in the circumstances on this section, as well. Our main result here is the following almost sure invariance principle:

\begin{theorem}\label{ASIP2}
Suppose that Assumptions \ref{Pairing ass} and \ref{Bound ass} hold true.
Assume also that there exists a constant $c>0$ so that for any sufficiently large $n$ and any $v\in\mathbb R^d$ we have
\begin{equation}\label{C}
\inf_{j}\text{Cov}_{\mu_j}(S_{j,n}) v\cdot v\geq cn|v|^2.
\end{equation}
Then  for any $\del>0$, there exists a coupling between $(u_j\circ T_0^j)$, considered as a sequence of random variables on the probability space $(\cX_0,\mu_0)$, and a sequence of centered Gaussian random vectors $Z_1,Z_2,...$ so that 
\[
\Big|S_n-\mu_0(S_n)-\sum_{j=1}^n Z_j\Big|=o(n^{\frac 14+\del}),\quad\text{almost surely}.
\]
 Moreover, with $\cS_n=S_n-\mu_0(S_n)$, there exists a constant $C>0$ so that for any unit vector $v\in\mathbb R^d$,
\begin{equation}\label{Var est}
\Big\|\cS_n\cdot v\Big\|_{L^2}-Cn^{\frac 14+\del}\leq \Big\|\sum_{j=1}^n Z_j\cdot v\Big\|_{L^2}\leq \Big\|\cS_n\cdot v \Big\|_{L^2}+Cn^{\frac 14+\del}.
\end{equation}
\end{theorem}
In Section \ref{CovSec} we will show that (\ref{C}) holds true when $T_j,\phi_j$ and $u_j$ are sufficiently small perturbations of a single map $T$, a function $\phi$ and a vector-valued function $u$, respectively.

\begin{remark}\label{Rem1}
We want to stress that in both Theorems \ref{ASIP 1} and \ref{ASIP2} the condition about the uniform growth rate of the covariances is, in principle, not needed in order to apply the spectral methods in \cite{GO} (see Theorem \ref{Gouzel Thm}).
In our circumstances, the simpler condition 
\begin{equation}\label{Simpler}
\text{Cov}_{\mu_0}(S_{0,n}) v\cdot v\geq cn|v|^2
\end{equation}
is enough, if we assume that
\begin{equation}\label{U}
\sup_{j,n}\sup_{|t|\leq r_0}\|\cL_{it}^{j,n}\|\leq C
\end{equation} 
 for some $r_0$ and $C$, where $\cL_{it}^{j,n}$  are the transfer operators $\cL_{it}^{j,n}$ defined at the beginning of Section \ref{Sec4}. Condition (\ref{Simpler}) holds true in the random dynamical system case considered in Section \ref{RDS}, but with a vector-valued random function $u_\om(x)=u(\om,x)$ which does not admit an $L^2(\mu)$ coboundary representation with respect to the skew product map at some direction $v\in\bbR^d\setminus\{0\}$ (all the notations are given in Section \ref{RDS}, see also the proof of Prposition 2 in \cite{YeorDavor}). 
 The reason we need (\ref{Uniform Cov}) and (\ref{C}) is that they guarantee (\ref{Simpler}). In fact, we show that such uniform lower bounds on the covariances yield exponentially fast decay as $n\to\infty$ of the norms in (\ref{U}), see Proposition \ref{Norm Prop}. For uniformly distance expanding maps a uniform bound on $\|\cL_{it}^{j,n}\|$ follows from the, so called, Lasota-Yorke inequality. This is not true in general for non-uniformly distance expanding maps, but when $L_j=1$ in Assumption \ref{Pairing ass} then a weak Lasota-Yorke inequality holds true which is still enough in order to derive (\ref{U}). In particular, when considering Manneville-Pomeau maps (as in Examples \ref{MP map} and \ref{Ex2.1}) our results hold true under (\ref{Simpler}), without assuming (\ref{Uniform Cov}) and (\ref{C}). Of course, many additional examples of interval and multidimensional maps can be given. 
 
 We also note that for the one dimensional results stated in Section \ref{RDS} we only need that the variance of $S_n$ grows linearly fast (i.e. that (\ref{Simpler}) holds true), which is satisfied when the function $u=u(\om,x)$ does not admit a coboundary representation with respect to the skew product map.
\end{remark}

\subsubsection{Examples}\label{Examples 2}

\begin{example}\label{Ex2.0}
For each $j$ let $g_j:\mathbb T^d\to\bbT^d$ be a linear expanding map. Then Assumptions \ref{Pairing ass} and \ref{Bound ass2} hold true 
when each $T_j$ is a map constructed as in Example \ref{Ex1.1} and $\phi_j$ has a sufficiently small oscillation (in contrast to Example \ref{Ex1.1} we do not require that the $T_j$'s lie in some neighborhood of a given map). In order to verify (\ref{C}) we will need at the end some condition of this form, but the results in Section \ref{RDS} do not require such assumptions.
\end{example}

\begin{example}\label{Ex2.1}
Suppose that $T_j=f_{\be_j}$ for some $\be_j\in(0,1)$, where $f_\be$ is the Manneville-Pomeau map from Example \ref{MP map}.
It is clear that Assumption \ref{Pairing ass} holds true. Moreover, Assumption \ref{Bound ass2} holds true, for instance, 
when $\phi_j=-t_j\log J(T_j)$ for sufficiently small $t_j$. By Proposition 5.3 in \cite{castro}, the transfer operators
\[
L_\beta g(x)=g(y(\beta,x))+g(\frac12(x+1))
\]
are continuous function of $\beta$ when considered as functions from $(0,1)$ to the space of linear operators acting on the space $(\cH,\|\cdot\|_{C^1})$, equipped with the operator norm.
This shows that  the conditions guaranteeing (\ref{C}) hold true when  all the $\be_j$'s are sufficiently close to a specific $\be$ (see Proposition \ref{CovProp}). We remark that for random Manneville-Pomeau maps (considered in Section \ref{RDS}) the lower bound (\ref{Simpler}) holds true when $u(\om,x)$ is not a coboundary at any direction, and so in these circumstances we get the ASIP for random compositions of the form $f_{\be(\te^{n-1}\om)}\circ\cdots\circ f_{\be(\te\om)}\circ f_{\be(\om)}$ where $\be(\om)$ is a random variables taking values at $(0,1)$ and $\te$ is some ergodic measure preserving system.
\end{example}

%Next, we can also consider more examples...interval maps, Pomeu...multidimensional Pomeu etc 

%Next, can also consider a sequential version of \cite{MSU}...

%\begin{remark}\label{Martingale Remark}
%When $\cX_j=M$ and  the functions $\phi_j=-\ln J(T_j)$ satisfy Assumption \ref{Bound ass}, where $J(T_j)$ is the Jacobian of the map $T_j$, then the normalized volume measure $m$ is preserved under each $\cL_0^{(j)}$. In this case we can apply Theorem 3.1 from \cite{Hyde} and derive an ASIP under the weaker assumption that $\text{var}_{\mu_0}(S_{0,n}u)\geq n^{\frac14+\del}$ for some $\del>0$ and any sufficiently large $n$ (this theorem uses reverse martingale approximation). Still, it is unclear under which conditions such a growth rate can be obtained in our context (expect in the circumstances of Theorem \ref{Variance SDS}), and these assumptions concerning the Jackobian are quite restrictive, and it is not clear which maps $T_j$ satisfy them (since now $\phi_j$ and the constants $d_j,L_j,q_j$ and $\sig_j$ depend on the map $T_j$). Moreover, martingale approximations do not yield the additional results obtained in Section \ref{Other} (while as a by product of the proof o Theorem \ref{ASIP} we obtain many additional statistical properties).
%\end{remark}

%In the case of the setup from Castro, the variance, in fact, will grow linearly in $n$ if $n$ is not a cobounary since the tranfer operator will be continuous...I have to cite something here...

%We need to talk here about the "Gibbs" measrue $\mu_j$
Example \ref{Ex2.1} is a specific case of the following 

\begin{example}
Assume that each $T_j$ is an interval map which is uniformly expanding only on some of its monotonicity intervals. Assumption \ref{Bound ass2} in these circumstances restricts the amount of contraction (in terms of the amount of expansion), in a uniform manner.
 A concrete example can be constructed as follows. Let $m_j\geq 2$ be a sequence of integers. For any $j$ define $T_j(x)=m_j x$ for $0\leq x<\frac1{m_j}$. On the interval $[\frac 1{m_j},1)$ we just assume that $T_j$ is one to one and onto $[0,1)$ and that derivative of $f_j$ from the right of $\frac 1{m_j}$ is smaller than $1$ but larger that $L_j^{-1}$ for some $L_j\geq 1$. Then we have $\sig_j=m_j$, $d_j=2$ and $q_j=1$. Now we can take a potential $\phi_j$ with a sufficiently small oscillation. It is clear that it is possible to construct examples with $d_j$ and $q_j$ depending on $j$. Higher dimensional analogous can be considered as well.
\end{example}

\begin{example}[An abstract topological example]
Let $\cX_j$ be a normalized compact metric space (i.e. with diameter $1$), and assume that there exist sequences $\eta_j\in(0,1)$ and $\gam_j>0$ so that $T_j B_j(z,\eta_j)=\cX_{j+1}$ and $\rho_{j+1}(T_j(z_1),T_j(z_2))\geq\gam_j(z_1,z_2)$ for any $j$ and $z,z_1,z_2\in\cX_j$ satisfying $\rho_j(z_1,z_2)<\eta_j$. Assumption \ref{Pairing ass} is satisfied if we assume that $\gam_j$ is bounded from below, and it takes values larger than $1$ on some ball of radius $\eta_j$ and values smaller than $1$ on another ball of the same radius. We note that, in fact, this example is more general that the other examples presented above.
\end{example}

%Theorem \ref{Variance SDS} (ii) is a consequence of Theorem \ref{stability}, together with the analyticity of the RPF triplets and Theorem \ref{MomThm.0} (iii) below, applied with $k=2$.

%The section below should contain at most 5 pages!

\section{A sequential RPF theorem via cones contractions}\label{Sec4}
In the setup of Section \ref{Model 1} for each $j$ we set $\cX_j=M$.
In both setups considered in Section \ref{Sec2}, 
for each $j\in\bbZ$ and $z\in\bbC$, consider the transfer operators $\cL^{(j)}_z$ which maps functions on $\cX_j$ to functions on $\cX_{j+1}$ by the formula
\begin{equation}\label{Tr op}
\cL^{(j)}_zg(x)=\sum_{y\in T_j^{-1}\{x\}}e^{\phi_j(y)+zu_j(y)}g(y).
\end{equation}
We also set $\cL_0^{(j)}=\cL^{(j)}$. For each $j,n$ and $z$ write
\[
\cL_z^{j,n}=\cL_z^{(j+n-1)}\circ\cdots\circ\cL_z^{(j+1)}\circ\cL_z^{(j)}.
\]
It is clear that $\cL_z^{(j)}\cH_j\subset \cH_{j+1}$. We will denote by $(\cL_z^{(j)})^*$ the appropriate dual operator.
Henceforth we will refer to $\sup_j\|u_j\|$ and  the constants in  Assumptions (H1), (H2) \ref{Pairing ass}, \ref{Bound ass} and \ref{Bound ass2} as the ``initial parameters".

\begin{theorem}\label{RPF SDS}
Let the sequence of maps $T_j$ satisfy the conditions from either Section \ref{Model 1} or Section \ref{Model 2}, where in the circumstances of Theorem \ref{ASIP 1} we also assume that $\ve$ appearing there is sufficiently small. 
Then there exists a neighborhood $U$ of $0$, which depends only on the initial parameters, so that for any $z\in U$
there exist families
$\{\la_j(z):\,j\in\bbZ\}$, $\{h_j^{(z)}:\,j\in\bbZ\}$ and $\{\nu_j^{(z)}:\,j\in\bbZ\}$ 
consisting of a nonzero complex number 
$\la_j(z)$, a complex function $h_j^{(z)}\in\cH_j$ and a 
complex continuous linear functional $\nu_j^{(z)}\in\cH_j^*$ such that:

(i) For any $j\in\bbZ$,
\begin{equation}\label{RPF deter equations-General}
\cL_z^{(j)} h_j^{(z)}=\la_j(z)h_{j+1}^{(z)},\,\,
(\cL_z^{(j)})^*\nu_{j+1}^{(z)}=\la_j(z)\nu_{j}^{(z)}\text{ and }
\nu_j^{(z)}(h_j^{(z)})=\nu_j^{(z)}(\textbf{1})=1
\end{equation} 
where $\textbf{1}$ is the function which takes the constant value $1$.
When $z=t\in\bbR$ then 
 $\la_j(t)>a$ and the function $h_j(t)$ takes values at some interval $[c,d]$, where $a>0$ and $0<c<d<\infty$ depend only on the initial parameters. Moreover, $\nu_j^{(t)}$ is a probability measure which assigns positive mass to open subsets of $\cX_j$ and the equality 
$\nu_{j+1}(t)\big(\cL_t^{(j)} g)=\la_j(t)\nu_{j}^{(t)}(g)$ holds true for any 
bounded Borel function $g:\cX_j\to\bbC$.

(ii) The maps 
\[
\la_j(\cdot):U\to\bbC,\,\, h_j^{(\cdot)}:U\to \cH_j\,\,\text{ and }
\nu_j^{(\cdot)}:U\to \cH_j^*
\]
are analytic and there exists a constant $C>0$, which depends only on the initial parameters
such that 
\begin{equation}\label{UnifBound}
\max\big(\sup_{z\in U}|\la_j(z)|,\, 
\sup_{z\in U}\|h_j^{(z)}\|,\, \sup_{z\in U}
\|\nu^{(z)}_j\|\big)\leq C,
\end{equation}
where $\|\nu\|$ is the 
operator norm of a linear functional $\nu:\cH_j\to\bbC$.
Moreover, there exist a constant $c>0$, which depends only on the initial parameters, so that $|\la_j(z)|\geq c$ 
and $\min_{x\in \cX_j}|h_j^{(z)}(x)|\geq c$ for any integer $j$ and $z\in U$.

(iii) There exist constants $A>0$ and $\del\in(0,1)$, 
which depend only on the initial parameters,
 so that for any $j\in\bbZ$, $g\in\cH_j$
and $n\geq1$,
\begin{equation}\label{Exponential convergence}
\Big\|\frac{\cL_z^{j,n}g}{\la_{j,n}(z)}-\nu_j^{(z)}(g)h^{(z)}_{j+n}\Big\|\leq A\|g\|\del^n
\end{equation}
where $\la_{j,n}(z)=\la_{j}(z)\cdot\la_{j+1}(z)\cdots\la_{j+n-1}(z)$. Moreover, the probability measures $\mu_j,\,j\in\bbZ$ given by $d\mu_j=h_j^{(0)}d\nu_j^{(0)}$ 
satisfy that $(T_j)_*\mu_j=\mu_{j+1}$ and that
 for any $n\geq1$ and $f\in\cH_{j+n}$,
\begin{equation}\label{ExpDec}
\big|\mu_j(g\cdot f\circ T_j^n)-\mu_j(g)\mu_{j+n}(f)\big|\leq A\|g\|\mu_{j+n}(|f|)\del^n.
\end{equation}
\end{theorem}
We want to stress that in the circumstances of Theorem \ref{ASIP2} we get Theorem \ref{RPF SDS} without the additional assumption that $T_j, \phi_j$ and $u_j$ lie in some neighborhood of $T,\phi$ and $u$, respectively. Such a condition is needed only in order to verify (\ref{C}), see Remark \ref{Rem1}. We also want to mention here that the measures $\nu_j$ are conformal in the sense of (\ref{Conformal}).

%We note that, as in \cite{Annealed}, 
%when the $\{T_j\}$ are "sequentially non-singular" the measures $\mu_j$ are absolutely continuous, then we have the following

%\begin{proposition}\label{NonSinThm}
%Let $\textbf{m}_j,\,j\in\bbZ$ be a family of probability measures on $\cE_j$, which assign positive mass to open sets, so that for each $j$
%we have $(T_j)_*\textbf{m}_j\ll\textbf{m}_{j+1}$ and that $e^{-f_j}=\frac{d(T_j)_*\textbf{m}_j}{d\textbf{m}_{j+1}}$. Then for any $j$ we have
%$\la_j(0)=1$ and $\nu_j^{(0)}=\textbf{m}_j$.
%\end{proposition}

The proof of Theorem \ref{RPF SDS} relies on the theory of real and complex cones. We will give a reminder of the appropriate  results concerning this theory  in the body of the proof of Theorem \ref{Complex cones Thm} below, and the readers are referred to Appendix A of \cite{book} for a summary of the main definitions and results concerning contraction properties of real and complex cones.

Theorem \ref{RPF SDS} follows  from the following

\begin{theorem}\label{Complex cones Thm}
There exist $r,d_0>0$ and a sequences $\cC_{j}$ of complex cones so that:

(i) The cones $\cC_j$ and their duals $\cC_j^{*}:=\{\nu\in\cH_j^*:\,\nu(c)\not=0\,\,\,\forall\nu\in\cC_j\setminus\{0\}\}$ have bounded aperture: there exists a constant $A>0$ and complex continuous linear functionals $a_j\in\cH_j^*$ and $b_j\in(\cH_j^*)^*$ so that for any $g\in\cC_j$ and $\la\in\cC_j^*$ we have
\[
\|g\|\leq A|a_j(g)|\,\,\text{ and }\,\,\|\la\|\leq C|b_j(\la)|.
\]

(ii) The cones $\cC_j$ are linearly convex, namely  for any $g\not\in\cC_j$
there exists $\mu\in\cC_j^*$ such that $\mu(g)=0$.

(iii) The cones $\cC_j$ are reproducing: there exist constants $k_0\in\bbN$  and $r_0>0$ so that for any $j$ and $g\in\cH_j$ there are $g_1,...,g_{k_0}\in\cC_j$ so that $g=g_1+...+g_{k_0}$ and 
\[
\|g_1\|+...+\|g_{k_0}\|\leq r_0\|g\|.
\]

(iv) For any $j\in\bbZ$,  and $z\in\bbC$ so that $|z|<r$ we have
\[
\cL_z^{j}\cC_j\subset\cC_{j+1}
\]
and the Hilbert diameter of the image with respect to the complex projective metric corresponding to the cone $\cC_{j+1}$ does not exceed $d_0$.
\end{theorem}

Relying on this theorem,  Theorem \ref{RPF SDS} follows exactly as in Chapters 4 and 5 \cite{book}. Indeed, the main assumption in Chapter 4 is the existence of families of cones satisfying all the properties described in Theorem \ref{Complex cones Thm}. Using these properties existence of RPF triplets $\la_j(z)$, $h_j^{(z)}$ and $\nu_j^{(z)}$ follows from general contraction properties of complex projective metrics. The analyticity of $\la_j(z)$, $h_j^{(z)}$ and $\nu_j^{(z)}$ in $z$ is guaranteed after the complex cone method is applied successfully since these triplets can be expressed as certain uniform limits of explicit expressions involving the transfer operators $\cL_z^{(j)}$, which are analytic in $z$. 

\begin{proof}[Proof of Theorem \ref{Complex cones Thm}]
Let $\del>0$ be so that $(1+\del)s<1$, where $s$ is defined in (\ref{s sup}), and let $\ka>0$ be so that 
$
\sup_j v(\phi_j)<\ka\del.
$
Consider the real cone
\[
\cC_{j,\bbR}=\{g\in\cH_j:\,g>0\,\text{ and }\,v(g)\leq\ka\inf g\}
\]
and let $\cC_j$ be its canonical complexification  which (see Appendix A in \cite{book}) is given by 
\begin{equation}\label{Complexification}
\cC_j=\{g\in \cH_j:\,\Re\big(\overline{\mu(g)}\nu(g)\big)
\geq0\,\,\,\,\forall\mu,\nu\in\cC_{j,\bbR}^*\}
\end{equation}
where $\cC_{j,\bbR}^*=\{\mu\in\cH_j^*:\,\mu(c)\geq 0\,\,\,\forall\,c\in\cC_{j,\bbR}\}$.

We begin with showing that the complex cones $\cC_j$ and their duals have bounded aperture. First, for any point $a\in \cX_j$ and $g\in\cC_{j,\bbR}$ we have
\[
\|g\|=\sup g+v(g)\leq \inf g+2v(g)\leq (1+2\ka)\inf g\leq (1+2\ka)g(a)
\]
where we used that $g(x)-g(y)\leq (\text{diam}(\cX_j))^\alpha v(g)\leq v(g)$ for any real-valued function on $\cX_j$. We conclude from Lemma 5.2 in \cite{Rug} that for any $g\in\cC_j$ we have
\[
\|g\|\leq 2\sqrt 2(1+2\ka)g(a)
\]
and therefore we can take $a_j(g)=g(a)$ for an arbitrary point $a\in\cX_j$. Next, in order to show that the cone $\cC_j$ has bounded aperture we will apply Lemma A.2.7 from \cite{book} which states that 
\[
\|\nu\|\leq M\nu(\textbf{1}),\,\,\forall\,\nu\in\cC_j^*
\]
if the complex cone $\cC_j$ contains the ball of radius $1/M$ around the constant function $\textbf{1}$. The first step in showing that such a ball exists is the following representation of the cone:
\[
\cC_{j,\bbR}=\cC_{j,\bbR,\ka}=\{g\in\cH_j: \,s_{x,y,t,\ka}(g)\geq 0,\,\,\,\,\forall\,\,(x,y,t)\in\Del_j\}
\]
where $\Del_j$ is the set of triplets $(x,y,t)\in\cX_j\times\cX_j\times\cX_j$ so that $x\not=y$ and 
\[
s_{x,y,t,\ka}(g)=\ka g(t)-\frac{g(x)-g(y)}{\rho_j^\al(x,y)}.
\]
Then (see Appendix A in \cite{book}), we  can write
 \begin{equation}\label{complexification}
\cC_j=\{x\in\cH_j:\,\Re\big(\overline{\mu(x)}\nu(x)\big)\geq0\,\,\,\,\,\forall\mu,
\nu\in \Del_j\}
\end{equation}
since $\Del_j$ generates the dual cone $\cC_{j,\bbR}^*$. Note that by Lemma 4.1 in \cite{Dub2}, 
a canonical complexification $\cC_\bbC$ 
of a real cone $\cC_\bbR$ is linearly convex if there exists
a continuous linear functional which is strictly positive on $\cC_\bbR'=\cC_\bbR\setminus\{0\}$. 

Using (\ref{complexification}), it is enough to find $\ve>0$ which does not depend on $j$ so that for any $g$ of the form $g=\textbf{1}+h$ with $\|h\|<\ve$, and any $(x_i,y_i,t_i)\in\Del_j$ for $i=1,2$, 
\[
\Re(s_{1}(g)\cdot\overline{s_2(g)})\geq 0
\]
where with $s_i=s_{x_i,y_i,t_i,\ka}$. This is indeed enough
 since then we can take $M=1/\ve$. Existence of such $\ve$ is clear since $s_i(g)=\ka-s_i(h)$ and $|s_i(h)|\leq (\ka+1)\|h\|$.  
 The cone $\cC_j$ is linearly convex since the real cone $\cC_{j,\bbR}$ has bounded aperture (so (ii) holds true).
 
 Now we will prove (iii). If $g\in\cH_j$ is real-valued then $g+c_g\in\cC_{j,\bbR}\subset\cC_j$ where $c_g=\max(\sup|g|,\nu(g)/\ka)$. It follows that $g=(g+c_g)-c_g$ is a sum of two members of $\cC_j$ so that 
 \[
 \|g+c_g\|+\|-c_g\|\leq 3(1+\ka^{-1})\|g\|.
 \]
 The proof of (iii) is completing by decomposing complex-valued functions $g$ in $\cH_j$ as $g=g_1+ig_2$ where $g_1,g_2\in\cH_j$ are real-valued. 
 
In order to prove (iv), we will first show that for any $j$,
\begin{equation}\label{Inclusion}
 \cL^{(j)}\cC_{j,\bbR}\subset\cC_{j+1,\bbR,\zeta\ka}
\end{equation}
where $\zeta=(1+\del)s<1$, where $\del$ was specified at the beginning of the proof of Theorem \ref{Complex cones Thm}.
In the setup of Section \ref{Model 1}, this was established for the transfer operator generated by $T$ and $\phi$ in the proof of Theorem 5.1 in \cite{castro}. According to Proposition 5.3 in \cite{castro} these transfer operators are continuous with respect to $C^1$-perturbations of $T$ and $\phi$, and therefore  (\ref{Inclusion}) holds true if $T_j$ and $\phi_j$ lie in a sufficiently small $C^1$-neighborhood of $T$ and $\phi$, respectively. 
In the setup of Section \ref{Model 2} we do not require that the $T_j$'s and the $u_j$'s lie in such a neighborhood, and instead the proof of (\ref{Inclusion}) proceeds similarly to the proof of Theorem 5.1 in \cite{castro}. Fix some $j$ and denote by $(x_i)$ and $(y_i)$ the inverse images of two points $x$ and $y$ under $T_j$, respectively. We have
\begin{eqnarray*}
\frac{|\cL^{(j)}g(x)-\cL^{(j)}g(y)|}{\inf\cL_0^{(j)}g}\\\leq
\frac{|\cL^{(j)}g(x)-\cL^{(j)}g(y)|}{d_j e^{\inf\phi_j}\inf g}\leq
d_j^{-1}\sum_{i=1}^{d_j}e^{\phi_j(x_i)-\inf\phi_j}|g(x_i)-g(y_i)|(\inf g)^{-1}\\+d_j^{-1}\sum_{i=1}^{d_j}|(g(y_i)/\inf g)e^{-\inf\phi_j}|e^{\phi_j(x_i)}-e^{\phi_j(x_i)}|:=I_1+I_2,
\end{eqnarray*}
where $\cL^{(j)}=\cL_0^{(j)}$.
Since $\rho_j(x_i,y_i)\leq L_j\rho_j(x,y)$ for any $1\leq i\leq q_j$ and $\rho_j(x_i,y_i)\leq \sig_j^{-1}\rho_j(x,y)$ for all other preimages, 
\begin{equation*}
I_1\leq \rho_{j+1}^\al(x,y)e^{\ve_j}d_j^{-1}(L_j^\al q_j+(d_j-q_j)\sig_j^{-\al})=\rho_{j+1}^\al(x,y)s\ka
\end{equation*}
where $s$ is defined in (\ref{s sup}), and we used that $|g(x_i)-g(y_i)|\leq v(g)\rho_j^\al(x_i,y_i)\leq\ka\inf g\cdot\rho_j^\al(x_i,y_i)$.

In order to bound $I_2$, we first observe that  $\sup g\leq\inf g+v(g)\leq (1+\ka)\inf g$ and that 
\[
|e^{\phi_j(x_i)}-e^{\phi_j(y_i)}|\leq e^{\max(\phi_j(x_i),\phi_j(y_i))}|\phi_j(x_i)-\phi_j(y_i)|
\leq e^{\inf\phi_j+\ve_j}v(\phi_j)\rho^\al_j(x_i,y_i).
\]
Using these estimates we obtain that 
\[
I_2\leq \rho_{j+1}^{\alpha}(1+\ka)s\cdot\sup_j v(\phi_j).
\]
We conclude that 
\[
v(\cL^{(j)}g)\leq s(\ka+\sup_j v(\phi_j))\inf\cL_0^{(j)}\leq s\ka(1+\del)\inf\cL_0^{(j)}=\zeta\inf\cL_0^{(j)}.
\]
and therefore
\begin{equation}\label{Real inv}
\cL^{(j)}\cC_{j,\bbR,\ka}\subset\cC_{j,\bbR,\zeta\ka}\subset\cC_{j,\bbR,\ka}.
\end{equation}
By Proposition 5.2 in \cite{castro} (see the proof of Proposition 4.3 from there), there exists $d_0$ which depends only on $\ka$ and $\zeta$ so that the real projective diameter of $\cC_{j,\bbR,\zeta\ka}$ as a subset of $\cC_{j,\bbR,\ka}$ does not exceed $d_0$.

We will next prove that  for any $j$, $\ka$, $(x,y,t)\in\Del_j$, $g\in\cC_{j,\bbR}$ and a complex $z$ so that $|z|\leq 1$ we have 
\begin{equation}\label{Comparison Key}
\big|s_{x,y,t,\ka}(\cL_z^{(j)})g-\cL_0^{(j)}g)\big|\leq c|z| s_{x,y,t,\ka}(\cL_0^{(j)}g)
\end{equation}
where $c$ is some constant which does not depend on $j$. After this is established we can apply Theorem A.2.4 from Appendix A in \cite{book}  and obtain item (iv).

We begin with the following simple result/observation: let $A$ and $A'$ be complex numbers, $B$ and $B'$ be real numbers, and let $\ve_1>0$ and $\zeta\in(0,1)$ so that
\begin{itemize}
\item
$B>B'$
\item
$|A-B|\leq\ve_1B$
\item
$|A'-B'|\leq\ve_1 B$
\item
$|B'/B|\leq\zeta$.
\end{itemize}
Then 
\[
\left|\frac{A-A'}{B-B'}-1\right|\leq 2\ve_1(1-\zeta)^{-1}.
\]
The proof of this results is elementary, just write
\[
\left|\frac{A-A'}{B-B'}-1\right|\leq\left|\frac{A-B}{B-B'}\right|+
\left|\frac{A'-B'}{B-B'}\right|\leq \frac{2B\ve_1}{B-B'}=\frac{2\ve_1}{1-B'/B}.
\]

Fix some nonzero $g\in\cC_{j,\bbR}$ and $(x,y,t)\in\Del_{j+1}$. We want to apply the above results with $A=\ka\cL_z^{(j)} g(t)$, 
\begin{equation*}
 B=\ka\cL_0^{(j)}g(t),\, A'=\frac{\cL_z^{(j)} g(x)-\cL_z^{(j)} g(y)}{\rho_j^\al(x,y)}\,\,
\text{ and }\,\,B'=
\frac{\cL_0^{(j)} g(x)-\cL_0^{(j)} g(y)}{\rho_j^\al(x,y)}.
\end{equation*}
We begin with noting that $B>B'$ since the function $\cL_0^{(j)}g$ is a nonzero member of the cone $\cC_{j,\bbR,\zeta\ka}$. 
Notice that when $|z|\leq 1$ we have
\begin{eqnarray*}
|A-B|=\ka|\cL_z^{(j)} g(t)-\cL_0^{(j)}g(t)|=\ka|\cL_0^{(j)}(g(e^{zu_j}-1))(t)|\\\leq 
\ka\|e^{zu_j}-1\|_\infty\cL_0^{(j)} g(t)\leq |z|e^{\|u_j\|_\infty}\|u_j\|_\infty B\leq C|z|B
\end{eqnarray*}
for some constant $C>0$, where we used that $\sup\|u_j\|_\infty<\infty$. Next, we have
\[
|B'/B|\leq \zeta\inf\cL_0^{(j)}g/B\leq\zeta<1
\]
where we  used that $\cL_0^{(j)}g$ is a nonzero member of the cone $\cC_{j,\bbR,\zeta\ka}$.
Finally, we estimate the difference $|A'-B'|$. For each $a,b\in\cX_j$ we define
\[
\Del_{a,b}(z)=e^{\phi_j(a)}(e^{zu_j(a)}-1)g(a)-e^{\phi_j(b)}(e^{zu_j(b)}-1)g(b).
\] 
Denote again by $x_i$ and $y_i$ the preimages of $x$ and $y$ under $T_j$, respectively, where $1\leq i\leq d_j$. Then 
\[
\rho_j^\alpha(x,y)(A'-B')=\sum_{i=1}^{d_j}\Del_{x_i,y_i}(z).
\]
We first have 
\[
|\Del_{a,b}(z)|=|\Del_{a,b}(z)-\Del_{a,b}(0)|\leq|z|\sup_{|q|\leq |z|}|\Del'_{a,b}(q)|
\]
where $\Del'_{a,b}(\cdot)$ is the gradient of $\Del_{a,b}(\cdot)$.
Next, since 
\[
|e^{\phi_j(a)}-e^{\phi_j(b)}|\leq (e^{\phi_j(a)}+e^{\phi_j(b)})v(\phi_j)\rho^\al_j(a,b)
\]
using that the sequence $(L_j)$ appearing in Assumption \ref{Pairing ass} is bounded, 
we obtain that for any $q$ such that $|q|\leq1$ and any $1\leq i\leq d_j$,
\[
|\Del'_{x_i,y_i}(q)|\leq CL_j(e^{\phi_j(x_i)}+e^{\phi_j(y_i)})\|g\|\rho_j^\al(x,y)
\]
where $C>0$ is some constant. We conclude that  there exists a constant $C>0$  so that for any $j\in\bbZ$ and $z\in\bbC^d$ with $|z|\leq1$,
\[
|A'-B'|\leq C|z|\|g\|_\infty(\cL_0^{(j)}\textbf{1}(x)+\cL_0^{(j)}\textbf{1}(y))\leq 
C_1|z|\inf g
\]
where we used that $\sup_{j}\|\cL_0^{(j)}\textbf{1}\|_\infty<\infty$. 
Since $\|\phi_j\|$ is bounded in $j$ there exists a constant $C_2>0$ so that
$\inf g\leq C_2\cL_0^{(j)}g(t)=C_2B$  for any $j$.  This completes the proof of (\ref{Comparison Key}). Applying Theorem A.2.4 in Appendix A of \cite{book} we complete the proof of (iv).
\end{proof}

%\subsection{Additional properties of the RPF triplets}

\section{Uniform control over the norms of $\cL_{it}^{j,n}$ and the covariance of $S_{j,n}$}

\subsection{Exponential decay of the norms}
We consider here the normalized operators  $\tilde\cL_z^{(j)}$ given by  $\tilde\cL_z^{(j)}(g)=\cL_z^{(j)}(gh_j)/h_{j+1}\la_j$ where $\la_j=\la_j(0)$ and $h_j=h_j^{(0)}$. Our proof of the ASIP will require the following
\begin{proposition}\label{Norm Prop}
Suppose that (\ref{C}) holds true. Then
there exist constants $r_0,c,C>0$ so that for any $t\in\mathbb R^d$ with $|t|\leq r_0$ and any $j$ and $n\geq0$  we have
\begin{equation}\label{Norm deacy}
\|\tilde \cL_{it}^{j,n}\|\leq Ce^{-c|t|^2n}.
\end{equation}
\end{proposition}

\begin{proof}
Without the loss of generality, we assume here that $\la_j(0)=1$ and $h_j^{(0)}\equiv \textbf{1}$, where $\textbf{1}$ is the function taking the constant value $1$ (we will use this notation regardless of the space this function is defined on). Otherwise we can just replace $\cL_z^{(j)}$ with $\tilde\cL_z^{(j)}$ (it is easy to find appropriate RPF triplets for $\tilde\cL_{z}^{(j)}$ using the ones corresponding to $\cL_z^{(j)}$, see for instance the arguments at the beginning of Section 4 in \cite{SeqRPF}). We will also assume without the loss of generality that 
\[
\mu_j(S_{j,n})=\sum_{k=0}^{n-1}\mu_{j+k}(u_{j+k})=0
\]
since otherwise we can replace $u_k$ with $u_k-\mu_k(u_k)$ for any $k$.

In these circumstances, it is clear that there exists $r>0$ so that on $\{z\in\bbC^d:\,|z|<r\}$ we can define functions $\Pi_k(z)$ which are uniformly bounded in $z$ and $k$ so that $\Pi_k(0)=0$ and $\la_k(z)=e^{\Pi_k(z)}$. Observe next that for each $j$ and $n$ we have
\begin{equation}\label{Basic rel}
\bbE e^{zS_{j,n}}=\mu_{j+n}(\cL_z^{j,n}\textbf{1})
\end{equation}
where $S_{j,n}=\sum_{k=j}^{j+n-1}u_j\circ T_j^n(\textbf{x})$ and $\textbf{x}$ is distributed according to $\mu_j$.
Using (\ref{Basic rel}) and (\ref{Exponential convergence}) we derive that when $r$ is sufficiently small then
\[
\big|\ln \bbE e^{zS_{j,n}}-\sum_{k=0}^{n-1}\Pi_{j+k}(z)\big|\leq c_2
\]
where $c_2$ is some constant which does not depend on $j,n$ and $z$. Since the expression inside the absolute value is analytic, we conclude by taking the second derivatives at $z=0$ and using the Cauchy integral formula that 
\begin{equation}\label{CovDiff}
\left|\text{Cov}(S_{j,n})-\text{Hessian}(\Pi_{j,n})|_{z=0}\right|\leq c
\end{equation}
where $\Pi_{j,n}(\cdot)=\sum_{k=0}^{n-1}\Pi_{j+k}(\cdot)$ and $c$ is some constant. On the other hand, 
it is also clear from (\ref{Exponential convergence}) that when $t\in\bbR^d$ has a sufficiently small length then for any $j$ and $n$ we have
\[
\|\cL_{it}^{j,n}\|\leq Ce^{\Re(\Pi_{j,n}(it))}.
\]
Next, using the relations in (\ref{RPF deter equations-General}), it follows that the gradient of $\Pi_k$ at $z=0$ equals $\mu_k(u_k)$ which we have assumed is $0$. Using now (\ref{CovDiff}) and the second order Taylor expansion of $\Pi_{j,n}$ around $0$, we conclude that there exist constants $r_0>0$ and $c>0$ so that for any $t\in\bbR^d$ with $|t|<r_0$ we have
\[
\left|\Pi_{j,n}(it)+\frac12\text{Hessian}(\Pi_{j,n})t\cdot t\right|\leq c|t|^3n.
\]
Using (\ref{CovDiff}) we get that 
\[
\left|\Pi_{j,n}(it)+\frac12\text{Cov}(S_{j,n})t\cdot t\right|\leq c|t|^3n+c_2|t|^2.
\]
By (\ref{C}) we have $\inf_j\text{Cov}(S_{j,n})t\cdot t\geq C n|t|^2$ for some $C>0$ and all sufficiently large $n$. Therefore, taking $r_0$ sufficiently to be small we deduce that if $|t|<r_0$ then for any $j$ and  sufficiently large $n$ we have 
\[
\Re(\Pi_{j,n}(it))\leq -c|t|^2 n
\]
where $c>0$ is some constant. We conclude that for such $t$'s we have
\begin{equation}\label{Norm deacy.1}
\|\tilde \cL_{it}^{j,n}\|=\|\cL_{it}^{j,n}\|\leq Ce^{-c|t|^2n}.
\end{equation}
\end{proof}

%It is possible to show that the measures $\mu_0$ satisfy a sequential version of the Gibbs property

\subsection{Strong stability and uniform  growth rate of the covariance matrix}\label{CovSec}
Let $T$, $\phi$ and $u$ be so that all of our conditions hold true with the ``sequence" $T_j=T$, $\phi_j=\phi$ and $u_j=u$. 
Let $\mu$ be the Gibbs measure corresponding to $T$ and $\phi$ which is obtained in Theorem \ref{RPF SDS}. In the setup of \cite{castro} $\mu$ is the unique equilibrium state corresponding to $T$ and $\phi$. We assume here without loss of generality that $\int u d\mu =0$. Using  the exponential decay of correlations (\ref{ExpDec}), it follows that the asymptotic covariance matrix 
\[
S^2=\lim_{n\to\infty}\frac 1n\text{Cov}_{\mu}(\sum_{k=0}^{n-1}u\circ T^k)
\]
exists and that it is positive definite if and only if there exists a non-zero $v\in\bbR^d$ so that the function $v\cdot u$ admits an $L^2(\mu)$ co-boundary representation $v\cdot u=r\circ T-r$. For each $z\in\mathbb C^d$ denote by $\cL_z$ the transfer operator generated by $T$ and the potential $e^{\phi+zu}$.

We have the following strong stability type result for the covariance of $S_{j,n}$:
\begin{proposition}\label{CovProp}
For any sufficiently small $r_0>0$ and $\del_0>0$ there exists $\ve>0$ with the following property: if 
\begin{equation}\label{Ve}
\sup_{j}\sup_{|z|\leq r_0}\|\cL^{(j)}_z-\cL_z\|\leq\varepsilon
\end{equation} 
then 
\[
\sup_j\left\|\text{Cov}_{\mu_j}(S_{j,n})-\text{Cov}_{\mu}(\sum_{k=0}^{n-1}u\circ T^k)\right\|\leq\del_0 n.
\]
In particular, when $S^2$ is positive definite then
there exists a constant $c>0$ so that for any sufficiently large $n$ and any $v\in\mathbb R^d$ we have
\[
\inf_{j}\text{Cov}_{\mu_j}(S_{j,n}) v\cdot v\geq cn|v|^2.
\]
\end{proposition} 
When $\cX_j=M$ are all the same  Riemannian manifold,  the maps $T_j$ satisfy the conditions from \cite{castro} and they lie in  $C^1$-ball of a single map satisfying these conditions, and the  functions $\phi_j$ and $u_j$ lie in some $C^1$-ball  around $\phi$ and $u$, respectively, then (\ref{Ve}) holds true with some $\ve$ which converges to $0$ when the radius of the latter $C^1$-ball converge to $0$ (see Proposition 5.3 in \cite{castro}). Of course, in these circumstances 
we consider the norm $\|g\|=\sup|g|+\sup|Dg|$.

Another example are intervals maps with finite number of monotonicity intervals which do not depend on $j$, where on each one of them each $T_j$ and $T$ are either expanding or contracting. If each $T_j$ is obtained from $T$ by  perturbing each inverse branch of $T$ in some H\"older norm, and $\phi_j$ and $u_j$ are small perturbations of $\phi$ and $u$ in this norm, then (\ref{Ve}) will hold true in the appropriate H\"older norm. Similar examples can be given for maps whose inverse branches are defined on certain rectangular regions in $\bbR^s$ for $s>1$. We also refer the readers to Section \ref{Examples 2}, in which condition (\ref{Ve}) is discussed in some of the examples given there.

\begin{proof}[Proof of Proposition \ref{CovProp}]
First, we claim that the arguments in the proof of Theorem 2.8 in \cite{SeqRPF} carry on exactly in the same way for vector valued functions $u_j$ (in \cite{SeqRPF} we considered scalar functions $u_j$). Indeed, these arguments only relied on the explicit limiting expressions of $\la_j(z),h_j^{(z)}$ and $\nu_j^{(z)}$ together with the analyticity of $z\to\cL_z^{(j)}$, two ingredients that we also have for vector valued functions $u_j$. 
It follows that the RPF triplets are strongly stable in the following sense. 
If $T_j,\phi_j$ and $u_j$ and $T_{1,j},\phi_{1,j}$ and $u_{1,j}$ are two triplets of sequences which satisfy the Assumptions from Section \ref{Sec2} with the same sequences $L_j,\sig_j,\ve_j$ etc.  then there exists a neighborhood $U_0\subset\bbC^d$ of $0$ so that for any $\delta>0$ there is an $\ve>0$ with the following property: if 
\begin{equation*}
\sup_{j}\sup_{z\in U_0}\|\cL^{(j)}_z-\cL_{1,z}^{(j)}\|\leq\varepsilon
\end{equation*} 
where $\cL_{1,z}^{(j)}$ are the transfer operators corresponding to the map $T_{1,j}$ and the potential $\phi_{1,j}+zu_{1,j}$, 
then 
\[
\sup_j\sup_{z\in U_0}\max\big(|\la_j(z)-\la_{1,j}(z)|,\|h_j^{(z)}-h_{1,j}^{(z)}\|,\|\nu_j^{(z)}-\nu_{1,j}^{(z)}\|\big)<\del.
\]
Here $\la_{1,j}(z)$, $h_{1,j}^{(z)}$ and $\nu_{1,j}^{(z)}$ are the triplets obtained in Theorem \ref{RPF SDS} for the transfer operators $\cL_{1,z}^{(j)}$.
In particular, it follows that 
\[
\sup_j\sup_{z\in U_0}|\Pi_{j}(z)-\Pi_{1,j}(z)|\leq c\delta
\]
where $c>0$ is some constant and $\Pi_{1,j}(z)$ is the pressure function corresponding to $\la_{1,j}(z)$. Applying this with $T_j=T$, $\phi_j=\phi$ and $u_j=u$ and using (\ref{CovDiff}) we complete the proof of the proposition.
\end{proof}

%remark that we only have examples in which we have linear growth rate...

\section{An almost sure vector-valued invariance principle}\label{Sec5}

%I can not do here the case theta<1, so I can just use the result I have with Davor...

Let $A_1,A_2,...$ be a sequence of $\bbR^d$-valued random vectors defined on some probability space$(\Om,\cF,P)$.
We recall the main  assumption from \cite{GO}, which was denoted there by (H). There exists $\ve_0>0$ and $C,c>0$ such that for any $n,m>0$, $b_1<b_2<...<b_{n+m+k}$, $k>0$ and $t_1,...,t_{n+m}\in\bbR^d$ with $|t_j|\leq\ve_0$, we have
\begin{eqnarray*}
\Big|\bbE\big(e^{i\sum_{j=1}^nt_j(\sum_{\ell=b_j}^{b_{j+1}-1}A_\ell)+i\sum_{j=n+1}^{n+m}t_j(\sum_{\ell=b_j+k}^{b_{j+1}+k-1}A_\ell)}\big)\\
-\bbE\big(e^{i\sum_{j=1}^nt_j(\sum_{\ell=b_j}^{b_{j+1}-1}A_\ell)}\big)\cdot\bbE\big(e^{i\sum_{j=n+1}^{n+m}t_j(\sum_{\ell=b_j+k}^{b_{j+1}+k-1}A_\ell)}\big)\Big|\\\leq C(1+\max|b_{j+1}-b_j|)^{C(n+m)}e^{-ck}.
\end{eqnarray*}

Our main results rely on the following modification of Theorem 1.3 in \cite{GO} which was proved
in \cite{YeorDavor}, Theorem 7:
\begin{theorem}\label{Gouzel Thm}
Suppose that $A_1,A_2,...$ is a sequence of centered random $d$ dimensional vectors satisfying Assumption (H) which is bounded in $L^p$ for some $p>4$. Assume, in addition, that there exist constants $a>0$ and $b\in\bbN$ so that for any $n\geq b$ and $v\in\mathbb R^d$
\begin{equation}\label{Go1}
\text{Cov}(\sum_{j=1}^{n}A_j)v\cdot v\geq an.
\end{equation}
Then for any $\ve>0$, there exists a coupling of $(A_j)_j$ with a sequence of independent centered Gaussian random vectors $(B_j)_j$ so that almost-surely as $n\to\infty$,
\[
\Big|\sum_{j=1}^{n}(A_j-B_j)\Big|=o(n^{a_p+\ve})
\]
where $a_p=\frac{p}{4(p-1)}=\frac 14+\frac1{4(p-1)}$.
 Moreover, there exists a constant $C>0$ so that for any unit vector $v\in\mathbb R^d$,
\begin{equation}\label{Var est}
\Big\|\sum_{j=1}^n A_j\cdot v\Big\|_{L^2}-Cn^{a_p+\ve}\leq \Big\|\sum_{j=1}^n B_j\cdot v\Big\|_{L^2}\leq \Big\|\sum_{j=1}^n A_j\cdot v \Big\|_{L^2}+Cn^{a_p+\ve}.
\end{equation}
\end{theorem}

We will show that the conditions of Theorem \ref{Gouzel Thm} hold true. Condition (\ref{Go1}) is guaranteed in the circumstances of both Theorem \ref{ASIP 1} and Theorem \ref{ASIP2}, and so we only need to show that condition (H) is satisfied. Consider again the
 transfer operators $\tilde\cL_{it}^{(j)}$ given by
\[
\tilde\cL_{it}^{(j)}g=\cL_{it}^{(j)}(gh_{j}^{(0)})/{\la_j(0)h_{j+1}^{(0)}}.
\]
Then $\tilde\cL_0\textbf{1}=\textbf{1}$, where $\textbf{1}$ is the function which takes the constant value $1$.
We will also write $\tilde\cL_{it}^{j,n}=\cL_{it}^{(j+n-1)}\circ\cdots\circ\cdots\cL_{it}^{(j+1)}\circ\cL_{it}^{(j)}$.
Then by (\ref{Norm deacy}), there exists $r_0$ and positive $c$ and $C$ so that for any $j$, $n$ and $t\in\bbR^d$ with $|t|\leq r_0$ we have
\begin{equation}\label{EST1.0}
\sup_{j,n}\|\tilde\cL_{it}^{j,n}\|\leq Ce^{-c|t|^2n}.
\end{equation}
In particular, there exists a constant $A>0$ so that 
\begin{equation}\label{EST1}
\sup_{|t|\leq r_0}\sup_{j,n}\|\tilde\cL_{it}^{j,n}\|\leq A.
\end{equation}
Moreover, 
 by (\ref{Exponential convergence}), we have
\begin{equation}\label{EST2}
\|\tilde\cL_{0}^{j,n}(g)-\mu_j(h)\|\leq C_2\|g\|\del^n 
\end{equation}
where $C_2$ is some constant and $\del\in(0,1)$. We recall that, for any $j$, $n$ and $t$ we have
\[
\mu_j (e^{it\sum_{k=0}^{n-1}u_{j+k}\circ T_j^k})=\mu_{j+n}(\tilde\cL_{it}^{j,n}\textbf{1}).
\]
We will also denote by $M_j$ the one dimensional projection given by $M_j(g)=\mu_j(g)\textbf{1}$.
Next, we assume without the loss of generality that $\mu_j(u_j)=0$ for any $j$. 
We will show next that condition (H) holds true with  $A_l=u_l\circ T_0^l$ and $\ve_0=r_0$.
%Now use the more advanced type of the LY type inequality...so we need to have 
 Indeed, for any choice of $t_i\in\bbR^d$ so that $|t_i|\leq r_0$, a finite sequence $(b_i)$ and $k>0$ we have 
\begin{eqnarray*} 
\mu_0(e^{i\sum_{j=1}^nt_j(\sum_{\ell=b_j}^{b_{j+1}-1}A_\ell)+i\sum_{j=n+1}^{n+m}t_j(\sum_{\ell=b_{j}+k}^{b_{j+1}+k-1}A_\ell)})\\=
\mu_{b_1}\Big(\big(\prod_{j=n+1}^{n+m}\tilde\cL_{it_j}^{b_{j+1}-b_j,b_j}\big)\circ\tilde\cL_{0}^{b_{n+1},k}\circ\big(\prod_{j=1}^n\tilde\cL_{it_j}^{b_j,b_{j+1}-b_j}\big)\textbf{1}\Big)
\\=
\mu_{b_1}\Big(\big(\prod_{j=n+1}^{n+m}\tilde\cL_{it_j}^{b_{j+1}-b_j,b_j}\big)\circ\big(\tilde\cL_{0}^{b_{n+1},k}-M_{b_{n+1}+k}\big)\circ\big(\prod_{j=1}^n\tilde\cL_{it_j}^{b_j,b_{j+1}-b_j}\big)\textbf{1}\Big)\\+
\mu_{b_1}\Big(\big(\prod_{j=n+1}^{n+m}\tilde\cL_{it_j}^{b_{j+1}-b_j,b_j}\big)\circ M_{b_{n+1}+k}\circ\big(\prod_{j=1}^n\tilde\cL_{it_j}^{b_j,b_{j+1}-b_j}\big)\textbf{1}\Big):=
I_1+I_2.
\end{eqnarray*}
Applying  (\ref{EST1}) and (\ref{EST2}) we derive that with some constant $C>0$ we have  $|I_1| \leq C^{n+m}\del^k$. Moreover, since $M_{j}=\mu_j\otimes \textbf{1}$ and $\tilde\cL_0^{(j)}\textbf{1}=\textbf{1}$ we have
\begin{eqnarray*}
I_2=\mu_{0}(e^{i\sum_{j=n+1}^{n+m}t_j(\sum_{\ell=b_{j}+k}^{b_{j+1}+k-1}A_\ell)})\cdot \mu_0(e^{i\sum_{j=1}^nt_j(\sum_{\ell=b_j}^{b_{j+1}-1}A_\ell)})
\end{eqnarray*}
which completes the proof that condition (H) holds true with the sequence $A_l$.

\section{Random dynamical systems: additional results}\label{RDS}
Let $(\Om,\cF,P,\te)$ be an ergodic and invertible measure preserving system. Let $\cX_\omega$ be a random compact subset of some compact metric space $\cX$ (see Chapter 5 of \cite{book}). We will consider here the case when $T_j=T_{\te^j\om}$ where $T_\om:\cX_\om\to\cX_{\te\om}$ is a random map so that the skew product $T(\om,x)=(\te\om,T_\om x)$ is measurable with respect to the restriction of the product $\sigma$-algebra $\cF\times\cB$ on the skew product space $\cE=\{(\om,x):\,\om\in\Om, x\in\cX_\om\}$, where $\cB$ is the Borel $\sig$-algebra on $\cX$. 
Let $\phi(\om,x)$ and $u(\om,x)$ be two measurable functions so that $\phi_\om(\cdot)=\phi(\om,\cdot)$ and $u_\om(\cdot)$ belongs to $\cH_j=\cH_{\te^j\om}$ and the norms $\|\phi_\om\|$ and $\|u_\om\|$ are bounded. In this case, for any $\om$ the map $T_\om$ satisfies Assumption \ref{Pairing ass} with constants $d_\om, L_\om, q_\om$ and $\sig_\om$ (instead of $d_j, L_j, q_j$ and $\sig_j$), and for the sake of simplicity we assume here that the first three random variables are bounded and that $\sig_\om-1$ is bounded from below.

When considering the maps $T_j=T_{\te^j\om}$ and the functions $\phi_j=\phi_{\te^j\om}$ and $u_j=u_{\te^j\om}$
 the RPF triplets have the form $\la_j(z)=\la_{\te^j\om}(z)$, $h_j^{(z)}=h_{\te^j\om}^{(z)}$ and $\nu_j^{(z)}=\nu_{\te^j\om}^{(z)}$, and they are measurable  in $\om$.
Set 
\[
S_n^\om u=\sum_{j=0}^{n-1}u_{\te^j\om}\circ T_{\te^{j-1}\om}\circ\cdots\circ T_{\te\om}\circ T_\omega
\]
and $d\mu_\om=h_\om^{(0)}d\nu_\om^{(0)}$. Then the measure $\mu:=\int \mu_\om dP(\om)$ is $T$-invariant, and it is possible to show that it is ergodic (see for instance the arguments in \cite{MSU}).
%\subsection{Other limit theorems for random dynamical systems}\label{Other}
We consider here only real-valued functions $u_\om$, though our method should yield results for vector-valued functions, as well. We also assume here that $\int u_\om d\mu_\om=0$ for $P$-a.a. $\om$, which is not really a restriction since we can alwyas repalce $u_\om$ with $u_\om-\int u_\om d\mu_\om$. By Theorem 2.3 in \cite{Kifer-1998}, there exists a number $\sig^2\geq0$ so that $P$-a.s. we have
\[
\sig^2=\lim_{n\to\infty}\frac1n\text{Var}_{\mu_\om}(S_n^\om u).
\]
Moreover, $\sig^2$ is positive if and only if $u=u(\om,x)$ does not admit a coboundary representation $u=r\circ T-r$, where $r\in L^2(\mu)$. We note that Theorem 2.3 in \cite{Kifer-1998} also yields that $S_n^\om u(x)/\sqrt n$ converges in distribution towards a centered normal random variable with variance $\sig^2$.

Our first result here is the following Berry-Esseen type theorem, which provides optimal convergence rate in the self-normalized version of the  central limit theorem proved in \cite{Kifer-1998}:
\begin{theorem}\label{Fiber BE}
Suppose that $\sig^2>0$. 
Then
there exists a random variable $c_\om$ such that for any $n\in\bbN$,
\begin{equation*}
\sup_{t}\big|\mu_\om\{x:\,S_n^\om u(x)\leq t\sig_{\om,n})-\Phi(t)\big|\leq c_\om n^{-\frac12}
\end{equation*}
where  $\Phi$ is the standard normal distribution function.
\end{theorem}
The proof of Theorem \ref{Fiber BE} proceeds similarly to Chapter 7 in \cite{book}, using the arguments in the proof of Proposition \ref{Norm Prop} with $j=0$.

Now we will discuss moderate and local large deviation type results.
First, since
\[
\mu_\om(e^{zS_n^\om u})=\mu_\om(\cL^{\om,n}_z(h_\om)/h_{\te^n\om}\la_{\om,n})
\]
where $\la_{\om,n}=\prod_{j=0}^{n-1}\la_{\te^j\om}(0)$, using (\ref{Exponential convergence}) 
we have that
\[
\lim_{n\to\infty}\frac 1n \mu_\om(e^{zS_n^\om u}) =\Pi(z):=\int\ln\la_\om(z)dP(\om)-\int\ln\la_\om(0)dP(\om)
\]
where $\Pi_\om(z)=\ln \la_\om(z)-\ln\la_\om(0)$.
Using that $\Pi_\om(z)$ in analytic in $z$, a standard application of the G\"arnder-Ellis theorem (see \cite{Dembo})  yields the following

\begin{theorem}\label{Gen: Var, LD, MD}
Suppose that all the above conditions hold true 
and that $\sig^2>0$.

(i) Then the following (optimal) moderate deviations principle holds true: 
for any strictly increasing sequence $(b_n)_{n=1}^{\infty}$ of real numbers 
so that $\lim_{n\to\infty}\frac{b_n}n=0$ and $\lim_{n\to\infty}\frac{b_n}{\sqrt n}=\infty$ and a
Borel set $\Gam\subset\bbR$ 
we have 
\begin{eqnarray}\label{MDP}
-\inf_{x\in\Gamma^0}I(x)\leq \liminf_{n\to\infty}\frac1{a_n^2}\mu_\om\{x: W_n^\om(x)\in\Gamma\}\,\,\text{and}\\
 \limsup_{n\to\infty}\frac1{a_n^2}\mu_\om\{x: W_n^\om(x)\in\Gamma\}\leq -\inf_{x\in\bar \Gamma}I(x)\nonumber
\end{eqnarray}
where
$W_n^\om=\frac{S_\om^nu-\mu_\om(S_{n}^\om u)}{b_n}$, $I(x)=-\frac{x^2}{2}$, $\Gam^o$ is the interior of $\Gamma$ and $\bar\Gamma$ is its closer.

(ii) Let $L(t)$ be the Legendre transform of $\Pi(t)$. Then, (\ref{MDP}) holds true for any Borel set 
 $\Gam\subset[\Pi'(-\del),\Pi'(\del)]$ with $W_n^\om=\frac{S_{0,n}u-\mu_0(S_{0,n}u)}{n}$ and $I(t)=L(t)$ (this is a local large deviations principle).
%for any $[a,b]\subset[\Pi'(-\del),\Pi'(\del)]$,
%\[
%\lim_{n\to\infty}\frac1{n}\ln\mu_0\big\{x\in\cE_0:\,\frac1{n}\bar S_{0,n}u(x)\in[a,b]\big\}=-\inf_{t\in[a,b]}L(t).
%\] 
%I need stuff to be convex..etc...can I guarantee it?take general measurable sets?
\end{theorem}
Observe that $\Pi'(-\del)<\Pi'(\del)$ when $\sig^2>0$ since then the function $t\to\Pi(t)$ is strictly convex
in some real neighbourhood of the origin.

\begin{remark}\label{Rem2}
In Chapter 7 of \cite{book} a local central limit theorem was derived for random uniformly distance expanding maps. When consider random non-uniformly expanding maps, the proof in \cite{book} proceeds in the same way under the assumption that for any compact set $J\subset\bbR$ there is a constant $C=C(J)>0$ so that $P$-a.s. for any $n\geq1$ we have
\[
\sup_{t\in J}\|\cL_{it}^{\te^{n-1}}\circ\cdots\circ\cL_{it}^{\te\om}\circ\cL_{it}^\om\|\leq C
\]
where $\cL_{z}^\om$ is the transfer operator generated by the map $T_\om$ and the potential $\phi_\om+zu_\om$. A mentioned in Remark \ref{Rem1}, for uniformly distance expanding maps such estimates follow from an appropriate Lasota-Yorke type inequality, but they also follow when $L_\om=1$ (using a weak Lasota-Yorke type inequality), and so we get the local central limit theorem, for instance, for random Manneville-Pomeau maps.
\end{remark}

\end{document}